\documentclass{amsart}
\usepackage[T1]{fontenc}
\usepackage{amsmath}
\usepackage{amsfonts} 
\usepackage{amsthm}
\usepackage{amssymb}
\usepackage{tikz-cd}
\usepackage{derivative}
\usepackage{dirtytalk}

\usepackage[backend=biber,
style=numeric,giveninits=true]{biblatex} 
\theoremstyle{plain}
\newtheorem{thm}{Theorem}[section]
\newtheorem{lem}[thm]{Lemma}
\newtheorem{obs}[thm]{Observation}

\newtheorem{cor}[thm]{Corollary}

\newtheorem*{claim*}{Claim}

\theoremstyle{definition}

\theoremstyle{remark}
\newtheorem{rem}{Remark}[thm]

\usepackage{enumitem}

\numberwithin{equation}{thm} 

\newlist{eqenum}{enumerate}{1}

\setlist[eqenum]{
  label=(\theequation),
  ref=\theequation,
  leftmargin=1.5cm,
  before=\let\olditem\item
         \renewcommand{\item}{\refstepcounter{equation}\olditem}
}

\newcommand{\R}{\mathbb{R}}
\newcommand{\C}{\mathbb{C}}
\newcommand{\K}{\mathbb{K}}

\newcommand{\Z}{\mathcal{Z}}

\newcommand{\ns}{\mathrm{ns}}
\newcommand{\sing}{\mathrm{sing}}
\newcommand{\Zar}{\mathrm{Zar}}

\newcommand{\OO}{\mathcal{O}}
\newcommand{\id}{\mathrm{id}}
\newcommand{\Emb}{\mathrm{Emb}}

\DeclareMathOperator{\Int}{int}
\DeclareMathOperator{\supp}{supp}
\title{Optimal embedding dimension in the Nash--Tognoli theorem}
\author{Juliusz Banecki}

\address{Faculty of Mathematics and Computer Science,
Jagiellonian University, ul. Lojasiewicza 6, 30-348 Krakow, Poland}
\email{juliusz.banecki@student.uj.edu.pl}

\begin{document}
\begin{abstract}
We prove that every smooth compact submanifold of $\R^n$ can be approximated up to a small isotopy by the real locus of a nonsingular complex algebraic subset of $\C^n$ defined over $\R$. This settles a version of a conjecture posed in 1952 by Nash. Moreover, we show that if the codimension of the manifold being approximated is at least two, then the approximating real algebraic sets can be chosen to all have the same predetermined biregular isomorphism type.  
\end{abstract}
\maketitle
\section{Introduction}
In this paper a smooth manifold is always meant to be without boundary, unless explicitly stated otherwise. It is not assumed connected. 

The problem of approximating smooth compact submanifolds of $\R^n$ by algebraic sets up to a small isotopy has a long and rich history. It began with the work of Seifert \cite{seifertAlgebraischeApproximationMannigfaltigkeiten1936}, who proved a positive result in the case of hypersurfaces. Later, in his 1952 paper \cite{nashRealAlgebraicManifolds1952} Nash was able to show that if $M\subset \R^n$ is a smooth compact connected manifold and $n>2 \dim M$, then $M$ can be approximated by a nonsingular connected component of a real algebraic set in $\R^n$. He conjectured that the extra connected components could be avoided, and that the restriction on the dimension of the ambient space was unnecessary. His paper became one of the motivating forces for the development of real algebraic geometry in the second half of the twentieth century. Twenty years later, in \cite{tognoliSuCongetturaDi1973}, Tognoli showed that indeed one can get rid of the unnecessary connected components, as long as the inequality $n>2 \dim M$ is satisfied. The bound was later improved to $(2n-1)/3>\dim M$ by Ivanov \cite{ivanovImprovementNashTognoliTheorem1984}. 

A significant improvement came in 1992 in the paper \cite{akbulutApproximatingSubmanifoldsAlgebraic1992} by Akbulut and King. Among other results concerning approximation of smooth submanifolds of general algebraic sets, they were able to prove the following two theorems:
\begin{thm}[{\cite[Theorem A]{akbulutApproximatingSubmanifoldsAlgebraic1992}}]\label{thm:AKWeak}
Let $M\subset \R^n$ be a smooth compact submanifold. Then it can be approximated by the nonsingular locus of a possibly singular real algebraic subset of $\R^{n}$. In particular, if $M$ is connected then it can be approximated by a nonsingular connected component of an algebraic subset of $\R^n$.
\end{thm}
\begin{thm}[{\cite[Theorem B]{akbulutApproximatingSubmanifoldsAlgebraic1992}}]\label{thm:AKStrong}
Let $M\subset \R^n$ be a smooth compact submanifold. Then the smooth submanifold $M\times\{0\}\subset \R^{n+1}$ of $\R^{n+1}$ can be approximated by a nonsingular real algebraic subset of $\R^{n+1}$.
\end{thm}
Here, a real algebraic subset of $\R^n$ is said to be nonsingular if its Zariski closure in $\C^n$ is smooth near its real locus. The approximation considered here is up to a small isotopy (see Section \ref{sec:pre} for details).

Despite these two promising results, Akbulut and King were not able to get rid of the additional dimension in Theorem \ref{thm:AKStrong}. To the author's best knowledge, there has not been any progress on this problem since then.

The main goal of this paper is to finally solve the approximation problem without increasing the ambient dimension. In fact, we are even able to settle it under the additional requirement that the entire complexification of the approximating algebraic set is smooth:
\begin{thm}\label{thm:NT}
Let $n\geq 0$ be an integer and let $M$ be a smooth compact submanifold of $\R^n$. Then $M$ can be approximated by the real locus of a nonsingular complex algebraic subset of $\C^n$ defined over $\R$.
\end{thm}
\begin{rem}
It is known that such a result is false if instead we consider approximation by the real loci of complex algebraic sets in $\C^n$ defined over $\R$, whose Zariski closures in $\C P^n$ are nonsingular \cite{akbulutTranscendentalSubmanifoldsRn1993}. 
\end{rem}
We derive Theorem \ref{thm:NT} from the following stronger result:
\begin{thm}\label{thm:main}
Let $m\geq 0$ be an integer and let $W\subset \R^m$ be a nonsingular compact real algebraic set. Let $n\geq \dim W+2$ be an integer. Suppose that we are given a smooth map $i:W\hookrightarrow \R^n$, which embeds $W$ into $\R^n$ as a smooth manifold. Then the map $i$ can be arbitrarily well approximated in the $\mathcal C^\infty$-topology by regular maps $j:W\rightarrow \R^n$ which satisfy the following conditions:
\begin{enumerate}
    \item the set $j(W)$ is a nonsingular algebraic subset of $\R^n$ and $j$ is a biregular isomorphism between $W$ and $j(W)$,
    \item the Zariski closure of $j(W)$ in $\C^n$ is a nonsingular complex algebraic set.
\end{enumerate}
\end{thm}
For the definitions of a regular map and a biregular isomorphism see Section \ref{sec:pre}.
\begin{proof}[Proof that Theorem \ref{thm:main} implies Theorem \ref{thm:NT}]
The conclusion of Theorem \ref{thm:NT} is trivial if $M$ is of codimension zero. If $M$ is of codimension one, then it is known that one can even approximate $M$ by the real locus of a complex algebraic set in $\C^n$ defined over $\R$, whose Zariski closure in $\C P^n$ is nonsingular \cite[Theorem 7.1]{bochnakVectorBundlesReal1989}.

Assume that $M$ is of codimension at least $2$. By the Nash--Tognoli theorem \cite{tognoliSuCongetturaDi1973} there exists an integer $m\geq 0$ and a nonsingular real algebraic set $W\subset \R^m$ diffeomorphic to $M$ via a diffeomorphism $i:W\rightarrow M$. It suffices to apply Theorem \ref{thm:main} and take $X:=j(W)$ as the approximation of $M$.
\end{proof}

Theorem \ref{thm:main} admits the following further consequences:
\begin{cor}
Let $m\geq0$ be an integer and let $W\subset \R^m$ be a nonsingular compact real algebraic set. Let $n\geq \dim W+2$ be an integer. Assume that $W$ can be embedded in $\R^n$ as a smooth manifold. Then $W$ can be embedded in $\R^n$ as a real affine variety in the sense of \cite{bochnakRealAlgebraicGeometry1998}. Moreover, in this case the embedding can be chosen in such a way that the copy of $W$ embedded in $\R^n$ has nonsingular Zariski closure in $\C^n$.
\end{cor}
\begin{proof}
Take $i:W\rightarrow \R^n$ as the smooth embedding of $W$ and apply Theorem \ref{thm:main}.
\end{proof}
\begin{cor}\label{cor:approxWithIsoType}
Let $m\geq 0,n\geq 2$ be integers. Let $M\subset \R^n$ be a smooth compact submanifold of $\R^n$ of codimension at least $2$, which is diffeomorphic to a nonsingular real algebraic set $W\subset \R^m$. Then $M$ can be approximated by algebraic subsets of $\R^n$ biregularly isomorphic to $W$. 
\end{cor}
\begin{proof}
Choose a diffeomorphism $i:W\rightarrow M$ and apply Theorem \ref{thm:main}.
\end{proof}
In \cite{nashRealAlgebraicManifolds1952}, Nash also considered the possibility of approximating smooth manifolds by nonsingular rational real algebraic sets. Corollary \ref{cor:approxWithIsoType} bears particular significance in this context:
\begin{cor}
Let $n\geq 2$ be an integer and let $M\subset \R^n$ be a smooth compact connected submanifold of codimension at least $2$. Then the following conditions are equivalent:
\begin{enumerate}
    \item $M$ can be approximated by rational nonsingular real algebraic subsets of $\R^n$, 
    \item $M$ is diffeomorphic to a rational nonsingular real algebraic set. \label{rationalDiffeo}
\end{enumerate}
\end{cor}
\begin{proof}
If $M$ can be approximated by rational nonsingular real algebraic subsets of $\R^n$, then choosing a sufficiently close approximation gives a rational nonsingular real algebraic subset of $\R^n$ diffeomorphic to $M$.

Conversely, suppose that $M$ is diffeomorphic to a rational nonsingular real algebraic set $W\subset \R^m$ for some $m\geq 0$. Then Corollary \ref{cor:approxWithIsoType} shows that $M$ can be approximated by algebraic subsets of $\R^n$ biregularly isomorphic to $W$. Since rationality is preserved under biregular isomorphisms, these approximating algebraic subsets are rational.
\end{proof}
\begin{rem}
Condition \eqref{rationalDiffeo} has been studied extensively in the literature. There is a complete classification of possible diffeomorphism types of nonsingular compact rational surfaces due to an old result of Comessatti \cite{comessattiSullaConnessioneSuperficie1914}. The situation is far more complicated in dimension three and higher; see \cite[Chapter 6]{mangolteRealAlgebraicVarieties2020} for a survey.
\end{rem}

We would like to emphasise that our current approach does not provide Theorem \ref{thm:main} and its consequences if $n=\dim W+1$. The author conjectures that these results are true in that case as well.

\section{Preliminaries}\label{sec:pre}
\subsection{Algebraic sets and maps between them}
Let us establish some terminology which will be used throughout the paper. All the definitions introduced below are classical in the complex case. In the real case they are also standard and compatible with the ones from the book \cite{bochnakRealAlgebraicGeometry1998}, which contains a comprehensive treatment of real algebraic geometry.

For $\K=\R$ or $\C$, an algebraic set is a subset $X$ of $\K^n$ for some $n$ given as the zero set of a family of polynomials with coefficients in $\K$. Algebraic subsets of $\K^n$ form the family of closed sets in a certain topology, called the Zariski topology. By the dimension of an algebraic set $X$ we mean its dimension as a Noetherian topological space with the Zariski topology induced from $\K^n$. 

We say that a complex algebraic set $X\subset \C^n$ is defined over $\R$, if its defining equations can be chosen with coefficients in $\R$ (or equivalently, if it is invariant under complex conjugation). We say that it is nonsingular at a point $x\in X$, if the germ of $X$ at $x$ is a germ of a smooth complex manifold of dimension equal to the dimension of $X$. We say that $X$ is nonsingular, if it is nonsingular at each of its points (note that this definition of nonsingularity requires $X$ to be equidimensional).

Given a real algebraic set $X\subset \R^n$ we denote by $X_\C\subset \C^n$ its Zariski closure in $\C^n$. We call $X_\C$ the complexification of $X$. The nonsingular locus $X^{\ns}$ of $X$ consists of those points $x\in X$, at which $X_\C$ is nonsingular. The set $X^\ns$ is an analytic submanifold of $\R^n$. We define $X^{\sing}:=X\backslash X^{\ns}$ and say that $X$ is nonsingular if $X^\sing=\emptyset$. 

We say that a continuous $\K$-valued function $f:X\rightarrow \K$ on a Zariski locally closed set $X\subset \K^n$ is polynomial if it is the restriction of a polynomial $F\in \K[x_1,\dots,x_n]$. We say that $f$ is regular, if for every point $x\in X$ there are polynomials $P,Q\in \K[x_1,\dots,x_n]$ such that $Q(x)\neq 0$ and $f=\frac{P}{Q}$ holds on a neighbourhood of $x$ in $X$ in the Zariski topology. Given another algebraic set $Y\subset \K^m$ for some $m\geq0$ and a continuous map $f:X\rightarrow Y$, we say that it is polynomial (resp. regular) if each of its coordinate functions $f_1,\dots,f_m$ is polynomial (resp. regular). We say that a regular map $f:X\rightarrow Y$ is a biregular isomorphism if it is bijective and its inverse is also regular. If $\K=\R$ and $X\subset \R^n$ and $Y\subset \R^m$ are algebraic sets, given a polynomial map $f:X\rightarrow Y$ by $f_\C$ we denote its unique extension to a polynomial map $f_\C:X_\C\rightarrow Y_\C$ between the complex algebraic sets $X_\C$ and $Y_\C$. We call $f_\C$ the complexification of $f$. 

We will often use the standard fact that, for morphisms of complex algebraic varieties, algebraic properness is equivalent to properness of the associated map in the Euclidean topology. In particular, a Euclidean-proper regular map is Zariski closed, and finite morphisms are Euclidean-proper; see \cite[Appendix B]{hartshorneAlgebraicGeometry1977}.

Throughout the paper for $n\geq 2$ we use the following notation for some coordinate projections on $\C^n$:
\begin{align*}
    \pi_1&:\C^n\rightarrow \C^{n-1}, &\pi_1(x_1,\dots, x_n)&:=(x_2,\dots,x_n), \\ 
    \pi_2&:\C^n\rightarrow \C^{n-1},\quad &\pi_2(x_1,\dots, x_n)&:=(x_1,x_3,\dots,x_n), \\ 
    \pi_{1,2}&:\C^n\rightarrow \C^{n-2},\quad &\pi_{1,2}(x_1,\dots, x_n)&:=(x_3,\dots,x_n). 
\end{align*}

\subsection{Smooth maps on locally closed sets}\label{sec:pre:smooth}
Let $U\subset \R^n$ be an open set. In this paper, the space $\mathcal C^\infty(U,\R^m)$ of smooth maps from $U$ into $\R^m$ is always considered with the weak topology, i.e. the topology of locally uniform convergence with derivatives up to arbitrarily high order (see \cite[p. 34]{hirschDifferentialTopology1976}). More generally let $T\subset \R^n$ be a locally closed subset of $\R^n$ and let $U$ be any open neighbourhood of $T$ with the property that $T$ is closed in $U$. By $\mathcal{C}^\infty(T,\R^m)$ we denote the set of all maps from $T$ into $\R^m$, which arise as restrictions of maps from $\mathcal{C}^\infty(U,\R^m)$ and we endow $\mathcal C^\infty(T,\R^m)$ with the quotient topology induced by the surjection $\mathcal{C}^\infty(U,\R^m)\rightarrow \mathcal{C}^\infty(T,\R^m)$. As in the proof of \cite[Claim 2.5]{bilskiApproximationPiecewiseregularMaps2020} one verifies easily that this definition does not depend on the choice of $U$. If $T\subset \R^n$ happens to be a smooth submanifold, then it is also clear that this definition is equivalent to the usual one. See \cite[Section 2.3.1]{bilskiApproximationPiecewiseregularMaps2020} for a different way to define this topology on $\mathcal C^\infty(T,\R^m)$ in the case that $T$ is compact, and for its basic properties. 

We will make use of the following fact:
\begin{obs}
Let $n,m\geq 0$ and let $T\subset \R^n$ be a compact set. Then the space of all smooth embeddings $\Emb(T,\R^m)$ of $T$ in $\R^m$, i.e. the space of all injective smooth functions $f\in \mathcal C^\infty(T,\R^m)$ whose inverses $f^{-1}:f(T)\rightarrow \R^n$ satisfy $f^{-1}\in \mathcal C^\infty(f(T),\R^n)$, is open in the $\mathcal C^\infty$-topology.
\end{obs}
\begin{proof}
Let $f\in \mathcal C^\infty(T,\R^m)$ be an embedding. We will show that the set $\Emb(T,\R^m)$ contains a neighbourhood of $f$ in $\mathcal C^\infty(T,\R^m)$.

Assume first that $n=m$ and that $f$ is the identity. Let $B_n\subset \R^n$ be a large closed ball containing $T$. Since $B_n$ is a smooth manifold with boundary, it is well known that the space $\Emb(B_n,\R^n)$ of smooth embeddings of $B_n$ into $\R^n$ is open in the $\mathcal C^\infty$-topology (\cite[p. 37, Theorem 1.4]{hirschDifferentialTopology1976}). By the proof of \cite[Claim 2.3]{bilskiApproximationPiecewiseregularMaps2020} we have that the linear surjection
\begin{equation*}
    \mathcal C^\infty(B_n,\R^n)\rightarrow \mathcal C^\infty(T,\R^n)
\end{equation*}
is an open map. Therefore, it maps the set $\Emb(B_n,\R^n)$ to a neighbourhood of the identity in $\mathcal C^\infty(T,\R^n)$. Since the restriction of an embedding is an embedding, we have that this neighbourhood is contained in $\Emb(T,\R^n)$.

Let us now consider the general case. Let $g\in \mathcal C^\infty(T,\R^m)$ be sufficiently close to $f$ in the $\mathcal C^\infty$-topology. By the definition of the $\mathcal C^\infty$-topology we then have that $g\circ f^{-1}$ is close to $f\circ f^{-1}=\id_{f(T)}$ in the $\mathcal C^\infty$-topology. By the already covered case we now know that $g\circ f^{-1}$ is an embedding, so $g=(g\circ f^{-1})\circ f$, being a composition of two embeddings, is an embedding as well.
\end{proof}

Given a smooth compact submanifold $M\subset \R^n$ and a family $\mathcal A$ of subsets of $\R^n$, we say that $M$ can be approximated by elements of $\mathcal A$ if for every neighbourhood $\mathcal U$ of the inclusion $i:M\hookrightarrow \R^n$ in the $\mathcal C^{\infty}$-topology there is a map $j\in \mathcal U$ such that $j(M)\in \mathcal A$. Note that if $\mathcal U$ is chosen path connected and sufficiently small so that it consists of embeddings then $j(M)$ has to be a smooth compact submanifold of $\R^n$ isotopic to $M$. 
\subsection{Preliminary results}
The proof of Theorem \ref{thm:main}, included in the next section, will be based on some results of Akbulut and King which we recall in the rest of this section.

The following lemma is a version of a result from \cite{akbulutApproximatingSubmanifoldsAlgebraic1992}. We find the proof given there somewhat vague, so we provide a more detailed one below:
\begin{lem}[{\cite[Theorem 8]{akbulutApproximatingSubmanifoldsAlgebraic1992}}]\label{lem:functionApproxAK}
Let $\theta:\C^m\rightarrow \C^n$ be a polynomial map defined over $\R$. Let $T$ be a compact subset of $\theta^{-1}(\R^n)$ invariant under complex conjugation and let $\psi:T\rightarrow \C$ be a continuous function, which satisfies $\psi(\bar z)=\overline{\psi(z)}$ for $z \in T$. Assume that $\theta\vert_T$ is finite-to-one. Then, there is a polynomial with real coefficients $\tilde \psi:\C^m\rightarrow \C$ such that $\tilde \psi\vert_T$ approximates $\psi$ arbitrarily closely in the $\mathcal C^0$-topology. 

Moreover, suppose that we are given a subset $P\subset T$ open in $T$, invariant under complex conjugation, and such that $\psi\vert_P$ is of class $\mathcal C^\infty$ as a map into $\C\cong \R^2$ and can be locally approximated in the $\mathcal C^\infty$-topology by restrictions of polynomials with real coefficients. Then, we may also assume that the map $\tilde \psi\vert_P$ is a $\mathcal C^\infty$-approximation of $\psi\vert_P$.
\end{lem}
\begin{proof}
Choose an arbitrary compact set $K\subset P$, which is invariant under complex conjugation. By the definition of the $\mathcal C^\infty$-topology on $\mathcal C^\infty(P)$, it suffices to show that there is a real polynomial $\tilde \psi$ such that $\tilde \psi\vert_T$ approximates $\psi$ arbitrarily closely in the $\mathcal C^0$-topology and $\tilde \psi\vert_K$ approximates $\psi\vert_K$ arbitrarily closely in the $\mathcal C^\infty$-topology.

Fix a point $x\in \theta(T)$ and a small open neighbourhood $x\in U_x\subset \R^n$. Let $A_x:=(\theta\vert_T)^{-1}(x)$ be the fibre over $x$. Since $\theta\vert_T$ is proper and $A_x$ is a finite discrete set we may assume that $(\theta\vert_T)^{-1}(U_x)$ splits as $U=\bigsqcup_{a\in A_x} U_x^a$, where $U_x^a$ is a neighbourhood of $a$ in $T$ and $U_x^{\bar a}=\overline{U_x^a}$ for $a \in A_x$. We may also assume that $U_x$ is so small that for every $a\in A_x$, if $a\in K$ then $U_x^a\subset P$ and $U_x^a\cap K=\emptyset$ otherwise. Take a generic projection $\pi:\C^m\rightarrow \C$ defined over $\R$. We may assume that it is injective on $A_x$, and that $U_x$ is so small that each set $\pi(U_x^a)$ is contained in the interior of a small closed ball $B_a$ centred at $\pi(a)$, such that the balls $(B_a)_{a\in A_x}$ are disjoint and satisfy $B_{\bar a}=\overline{B_a}$ for $a\in A_x$. Let $B^x:=\bigsqcup_{a\in A_x}B_a$.

Fix any $a\in A_x\backslash K$. Let $f_a: B^x\rightarrow \C$ be given by 
\begin{equation*}
    f_a(z):=\begin{cases}
        \psi(a)& \text{if }z\in B_a \\
        \psi(\bar a)& \text{if }z\in B_{\bar a} \\
        0 & \text{otherwise}.
    \end{cases}
\end{equation*}
Applying Mergelyan's theorem (\cite[Theorem 20.5]{rudinRealComplexAnalysis1974}) we find a complex polynomial $\tilde f_a\in \C[t]$, such that $\tilde f_a\vert_{B^x}$ is close to $f_a\vert_{B^x}$ in the $\mathcal C^0$-topology. Using Cauchy's estimate we then deduce that $\tilde f_a\vert_{\Int B^x}$ is close to $f_a\vert_{\Int B^x}$ in the $\mathcal C^\infty$-topology. After considering $\frac{1}{2}\left(\tilde f_a(t)+\overline{\tilde f_a(\bar t)}\right)$ instead, we can assume that $\tilde f_a$ is defined over $\R$. Define the polynomial $g_a\in \R[x_1,\dots,x_m]$ as $g_a:=\tilde f_a\circ \pi$.

Fix now any $a\in A_x\cap K$. Let $f_a:B^x\rightarrow \C$ this time be given by 
\begin{equation*}
    f_a(z):=\begin{cases}
        1 & \text{if }z\in B_a\cup B_{\bar a} \\
        0 & \text{otherwise}.
    \end{cases}
\end{equation*}
Applying Mergelyan's theorem we again find a polynomial $\tilde f_a\in \R[t]$, such that $\tilde f_a\vert_{\Int B^x}$ is close to $f_a\vert_{\Int B^x}$ in the $\mathcal C^\infty$-topology. If $U_x$ is small enough, then by assumption $\psi\vert_{U_x^a}$ can be approximated by restrictions of polynomials with real coefficients. Choose such a polynomial $h_a\in \R[x_1,\dots,x_m]$. Then define $g_a:=h_a\cdot(\tilde f_a\circ \pi)$.

Let $\eta_x\in \R[x_1,\dots,x_m]$ be given by
\begin{equation*}
    \eta_x:=\sum_{a\in A'_x}g_a,
\end{equation*}
where $A'_x\subset A_x$ is any subset such that $|\{a,\bar a\}\cap A'_x|=1$ for all $a\in A_x$.

Choose a covering of $\theta(T)$ by sets of the form $U_x$ for $x\in I$, where $I\subset \theta(T)$ is a finite set. Choose a smooth partition of unity subordinate to this covering, i.e. a family of functions $\big(\varphi_x:\R^n\rightarrow [0,1]\big)_{x\in I}$ such that $\supp \varphi_x \subset U_x$ and $\sum_{x\in I}\varphi_x\equiv 1$ in a neighbourhood of $\theta(T)$. Using the Stone-Weierstrass theorem with derivatives (see \cite[Theorem 8.8.5]{bochnakRealAlgebraicGeometry1998}) we find real polynomials $\tilde \varphi_x\in \R[y_1,\dots,y_n]$ which are $\mathcal C^\infty$-close to the respective $\varphi_x$ near $\theta(T)$ in $\R^n$. Finally, define
\begin{equation*}
    \tilde \psi(z):=\sum_{x\in I} \tilde \varphi_x(\theta(z))\eta_x(z).
\end{equation*}
If the sets $U_x$ for $x\in \theta(T)$ were chosen small enough, and the finitely many approximating functions 
\begin{equation*}
    (\tilde \varphi_x)_{x\in I},(\tilde f_a)_{a\in A_x,x\in I},(h_a)_{a\in A_x\cap K,x\in I}
\end{equation*}
are chosen sufficiently well, then $\tilde \psi\vert_T$ is $\mathcal C^0$-close to $\psi$. Moreover, once we fix the covering $(U_x)_{x\in I}$, if the approximating functions are chosen well enough then $\tilde \psi\vert_K$ is $\mathcal C^\infty$-close to $\psi\vert_K$ as well.
\end{proof}

We will also make use of the following result from \cite{akbulutApproximatingSubmanifoldsAlgebraic1992}:
\begin{lem}[{\cite[Lemma 10]{akbulutApproximatingSubmanifoldsAlgebraic1992}}]\label{lem:immersion+injection=>nonsingular}
Let $X\subset \R^n$ be a real algebraic set. Let $f:X\rightarrow \R^m$ be a polynomial map and define $Y:=\overline{f(X)}^\Zar\subset \R^m$. Assume that the following conditions hold true:
\begin{enumerate}
    \item the complexification $f_\C:X_\C\rightarrow \C^m$ is a finite map,
    \item the set $X^\ns$ is compact and $f\vert_{X^\ns}$ is an embedding of $X^{\ns}$ as a smooth manifold,
    \item $(f_\C)^{-1}(f(X^\ns))=X^{\ns}$.
\end{enumerate}
Then $Y^{\ns}=f(X^\ns)$.
\end{lem}

The conclusion of Lemma \ref{lem:immersion+injection=>nonsingular} can be strengthened as explained in the following observation:
\begin{obs}\label{obs:strengthening}
In the conclusion of Lemma \ref{lem:immersion+injection=>nonsingular}, the map $f$ induces a biregular isomorphism between $X^\ns$ and $Y^\ns$.
\end{obs}
\begin{proof}
We may assume that $X$ is irreducible; the general case then follows by applying the conclusion to each of the irreducible components of $X$ separately. Fix a point $x\in X^\ns$ and let $y:=f(x)$. By the conclusion of Lemma \ref{lem:immersion+injection=>nonsingular} we know that $y$ is a nonsingular point of $Y$. Let $A:=\OO_{Y_\C}$ and $B:=\OO_{X_\C}$ denote the coordinate rings of $Y_\C$ and $X_\C$ respectively. Denote by $\mathfrak m\subset A$ and $\mathfrak n\subset B$ the maximal ideals of $A$ and $B$ corresponding to the points $y$ and $x$ respectively.  We will show that $f_\C$ induces an isomorphism between the local rings $A_{\mathfrak m}$ and $B_{\mathfrak n}$. From this the conclusion will follow, since the local regular inverse of $f_\C$, if exists, must be defined over $\R$.

Since $f_\C$ is a dominant map, the morphism $f_\C^\ast:A\rightarrow B$ is injective. It allows us to identify $A$ with a subring of $B$. In turn the induced map between the local rings $A_{\mathfrak m}$ and $B_{\mathfrak n}$ is injective too, so it remains to show that it is surjective.
\begin{claim*}
We have $\mathfrak mB=\mathfrak n$.
\end{claim*}
\begin{proof}
Obviously $\mathfrak mB\subset\mathfrak n$, so it suffices to prove the other inclusion.

The map $f$ is an immersion at $x$, so its derivative is of rank $\dim X$. Hence, the derivative of $f_\C$ at $x$ considered as a linear map over $\C$ is also of rank $\dim X_\C=\dim X$, where by $\dim X_\C$ we mean its complex dimension. It follows that it induces an isomorphism between the tangent spaces to $X_\C$ at $x$ and to $Y_\C$ at $y$. From this we deduce that
\begin{equation*}
    \mathfrak nB_{\mathfrak n}=\mathfrak m B_{\mathfrak n}+\mathfrak n^2B_{\mathfrak n}.
\end{equation*}
Using Nakayama's lemma we find that in fact $\mathfrak nB_{\mathfrak n}=\mathfrak m B_{\mathfrak n}$. This means that there is an element $u\in B,u\not \in \mathfrak n$ such that 
\begin{equation}\label{eq:obsProof1}
    u\mathfrak n\subset \mathfrak mB.    
\end{equation}

On the other hand, by assumption, $(f_\C)^{-1}(f(X^\ns))=X^\ns$, and $f|_{X^\ns}$ is
injective. Hence $x$ is the only point of $X_\C$ lying in the fibre over $y=f(x)$, so the radical of $\mathfrak mB$ is $\mathfrak n$. Since $\mathfrak n$ is a maximal ideal this means that $\mathfrak m B$ is $\mathfrak n$-primary. From \eqref{eq:obsProof1} we deduce that $\mathfrak n\subset \mathfrak m B$, which finishes the proof of the claim.
\end{proof}
Since the map $f_\C$ is finite by assumption, it turns $B$ into a finite $A$-module. From the claim we infer that
\begin{equation*}
    B=A+\mathfrak m B\text{ as an $A$-module.}
\end{equation*}
Indeed, for every \(b\in B\), we have \(b-b(x)\in\mathfrak n=\mathfrak mB\), while \(b(x)\in\C\subset A\). Invoking Nakayama's lemma once more we find an element $r\in A$ such that $r\not \in \mathfrak m$ and $rB\subset A$. Now, if $b\in B_{\mathfrak n}$ is written as $b=\frac{p}{q}$ with $p,q\in B$ and $q\not \in \mathfrak n$ then it may be alternatively written as
\begin{equation*}
    b=\frac{p}{q}=\frac{pr}{qr},
\end{equation*}
with $pr,qr\in A$. Moreover, $(qr)(y)=q(x)r(y)\neq 0$, so $qr$ is invertible in $A_\mathfrak m$. Hence $b\in A_\mathfrak m$, which proves that the induced map of local rings is surjective.
\end{proof}

\section{Proof of Theorem \ref{thm:main}}\label{sec:mainProof}
We will gradually make our way towards Theorem \ref{thm:main} via several lemmas. The proofs of Lemmas \ref{lem2}, \ref{lem3} and Theorem \ref{thm:main} are based on a \say{singularities shifting} technique developed by Akbulut and King in their proof of Theorems \ref{thm:AKWeak} and \ref{thm:AKStrong}. Interestingly, in the course of our proof we switch the roles of the first two variables twice, so that we first shift the singularities in the direction of the variable $x_1$, then in the direction of $x_2$, and then we shift them in the direction of $x_1$ again, this time to infinity.

The following lemma is a variant of \cite[Lemma 6]{akbulutApproximatingSubmanifoldsAlgebraic1992}:
\begin{lem}\label{lem1}
Let $W\subset \R^m$ be a nonsingular compact algebraic set. Let $i:W\rightarrow \R^n$ be a smooth map, which embeds $W$ in $\R^n$ as a smooth manifold. Then, there is a polynomial map $j:W\rightarrow \R^n$, which is arbitrarily close to $i$ in the $\mathcal C^\infty$-topology and such that the complexified map $j_\C:W_\C\rightarrow \C^n$ is finite.
\end{lem}
\begin{proof}
Thanks to the Stone-Weierstrass theorem, we can find a polynomial map $f:W\rightarrow \R^n$ close to $i$ in the $\mathcal C^\infty$-topology. Let $V=\{(x,y)\in \R^m \times \R^n:y=f(x)\}$ be the graph of $f$ and 
\begin{equation*}
    g:W\rightarrow V,\quad g(x):=(x,f(x)).
\end{equation*}
The complexified map $g_\C:W_\C\rightarrow V_\C$ is an isomorphism, so in particular it is finite. Due to Noether's projection lemma we may find a generic projection $\pi:\C^{m+n}\rightarrow \C^n$ defined over $\R$, close to the standard projection onto the last $n$ coordinates and such that $\pi\vert_{V_\C}$ is finite. Then $j:=\pi\circ g$ satisfies the desired properties.
\end{proof}

The proof of the next lemma is adapted from the proof of \cite[Theorem A]{akbulutApproximatingSubmanifoldsAlgebraic1992}.
\begin{lem}\label{lem2}
Let $W\subset \R^m$ be a nonsingular compact algebraic set. Let $i:W\rightarrow \R^n$ be a smooth map, which embeds $W$ in $\R^n$ as a smooth manifold. Assume that $n\geq \dim W+2$. Then, there is a polynomial map $j:W\rightarrow \R^n$ arbitrarily close to $i$ in the $\mathcal C^\infty$-topology, such that if we define $X:=\overline{j(W)}^\Zar$ then the following properties are satisfied:
\begin{eqenum}
    \item $X^\ns=j(W)$ and $j$ is a biregular isomorphism between $W$ and $X^\ns$, \label{lem2cond1}
    \item the map $\pi_{1,2}\vert_{X_\C}:X_\C\rightarrow \C^{n-2}$ is finite, \label{lem2cond2}
    \item $\pi_2(X^\sing)\cap \pi_2(X^\ns)=\emptyset$. \label{lem2cond3}
\end{eqenum}
\end{lem}
\begin{proof}
We will construct a polynomial map $j:W\rightarrow \R^n$ arbitrarily close to $i$ in the $\mathcal C^\infty$-topology, such that the following two conditions are satisfied:
\begin{eqenum}
    \item the composition $\pi_{1,2}\circ j_\C:W_\C\rightarrow \C^{n-2}$ is a finite map, \label{lem2cond1'}
    \item if $w\in W_\C$ is such that $j_\C(w)\in \R^n$ and $\pi_2\circ j_\C(w)\in \pi_2\circ j(W)$, then $w\in W$. \label{lem2cond2'}
\end{eqenum}
Let us see how \eqref{lem2cond1}-\eqref{lem2cond3} follow from \eqref{lem2cond1'} and \eqref{lem2cond2'}.

First of all, since the finite map $\pi_{1,2}\circ j_\C:W_\C\rightarrow \C^{n-2}$ factors as
\begin{equation*}
    W_\C\xrightarrow{j_\C}X_\C\xrightarrow{\pi_{1,2}}\C^{n-2},
\end{equation*}
where the first map is dominant, we have that $\mathcal O_{X_\C}$ is a $\C[x_3,\dots,x_n]$-submodule of the finite $\C[x_3,\dots,x_n]$-module $\OO_{W_\C}$, so it is itself finite. This proves \eqref{lem2cond2}. Moreover, a fortiori we have that $\OO_{W_\C}$ is a finite $\OO_{X_\C}$-module, so $j_\C$ is finite. 

Notice that \eqref{lem2cond2'} implies that $j_\C^{-1}(j(W))=W$. Indeed, if $w\in W_\C$ satisfies $j_\C(w)\in j(W)$, then $j_\C(w)\in \R^n$ and $\pi_2\circ j_\C(w)\in \pi_2\circ j(W)$, so $w\in W$. This, together with the assumption that $j$ is a $\mathcal C^\infty$-approximation of $i$, shows that the assumptions of Lemma \ref{lem:immersion+injection=>nonsingular} are satisfied, so Lemma \ref{lem:immersion+injection=>nonsingular} and Observation \ref{obs:strengthening} imply that \eqref{lem2cond1} holds true.

Finally, let us show why \eqref{lem2cond3} follows from \eqref{lem2cond1'} and \eqref{lem2cond2'}. Assume that there is some $x\in X^\sing$ such that $\pi_2(x)\in \pi_2(X^\ns)$. Since the map $j_\C:W_\C\rightarrow X_\C$ is finite and dominant, we have that it is surjective. Let $w\in W_\C$ be such that $j_\C(w)=x$. We have that $j_\C(w)\in X\subset \R^n$ and $\pi_2\circ j_\C(w)=\pi_2(x)\in \pi_2(X^\ns)=\pi_2\circ j(W)$. From \eqref{lem2cond2'} it follows that $w\in W$, so $x\in j(W)=X^\ns$. This is a contradiction, which proves \eqref{lem2cond3}.

It remains to construct a polynomial approximation of $i$ satisfying \eqref{lem2cond1'} and \eqref{lem2cond2'}. Applying Lemma \ref{lem1} we may replace $i$ by an arbitrarily close map and assume that $i$ itself is a polynomial map whose complexification $i_\C:W_\C\rightarrow \C^n$ is finite. Since $n\geq \dim W+2$, after performing a generic linear change of coordinates close to the identity, we may assume that the composition $\pi_{1,2}\circ i_\C:W_\C\rightarrow \C^{n-2}$ is also finite.

Define the set 
\begin{equation*}
    T:=(\pi_{1,2}\circ i_\C)^{-1}(\pi_{1,2}\circ i(W))\subset W_\C.
\end{equation*}
Since $\pi_{1,2}\circ i_\C$ is proper and $W$ is compact, we have that $T$ is a compact subset of $W_\C$ containing $W$ and invariant under conjugation. Due to the fact that $i_\C$ is an immersion at points of $W$ and embeds $W$ in $\R^n$, we may find a closed neighbourhood $T_1\subset T$ of $W$ such that $i_\C\vert_{T_1}$ is an embedding of $T_1$ in $\C^n$. After replacing $T_1$ by $T_1\cap \overline{T_1}$, we may assume that $T_1$ is invariant under conjugation. By a similar argument we can then find a smaller closed neighbourhood $T_2\subset \Int T_1$ of $W$ in $T$, contained in the interior of $T_1$ in $T$ and invariant under conjugation.

Since $T$ is compact, there is a constant $C>0$ such that 
\begin{equation}
    |(i_\C(w))_1|<C\text{ for }w\in T, \label{eqC}
\end{equation}
where $(i_\C(w))_1$ denotes the first coordinate of the point $i_\C(w)\in \C^n$. 

Since $i_\C\vert_{T_1}$ is an embedding, it is in particular injective. Let $w\in T_1\backslash \Int T_2$, where the interior is meant in the Euclidean topology on $T$. If we assume that $i_\C(w)\in \R^n$ then $i_\C(\bar w)=\overline{i_\C(w)}=i_\C(w)$, so $w=\bar w$ and hence $w\in W\subset \Int T_2$. This is a contradiction, so we must have $i_\C(T_1\backslash \Int T_2)\cap \R^n=\emptyset$. As the set $T_1\backslash \Int T_2$ is compact, there is a constant $\delta>0$ such that 
\begin{equation}\label{eq:imIneq}
    \Vert \Im(i_\C(w))\Vert>\delta\text{ for }w\in T_1\backslash T_2,
\end{equation}
where, for $v\in \C^n$, $\Im(v)\in \R^n$ denotes the vector consisting of imaginary parts of the coordinates of $v$.

By means of Urysohn's lemma we construct a continuous function $\psi:T\rightarrow \R$, which is constantly equal to zero on an open neighbourhood $P$ of $T_2$ in $T$ and which is constantly equal to $4C$ on $T\backslash T_1$. After substituting $\frac{1}{2}(\psi(w)+\psi(\bar w))$ for $\psi(w)$, we may assume that it satisfies $\psi(\bar w)=\psi(w)=\overline{\psi(w)}$. 

Let $\theta:\C^m\rightarrow \C^{n-2}$ be any polynomial extension of $\pi_{1,2}\circ i_\C$ defined over $\R$. By the definition of $T$ we have $T\subset \theta^{-1}(\R^{n-2})$, and since $\pi_{1,2}\circ i_\C:W_\C\rightarrow \C^{n-2}$ is finite we have that $\theta\vert_T$ is finite-to-one. This allows us to apply Lemma \ref{lem:functionApproxAK} and find a polynomial $\tilde \psi:\C^m\rightarrow \C$ defined over $\R$ such that 
\begin{equation}\label{lem2eqApprox}
    |\tilde \psi(w)-\psi(w)|<\min(\delta,C)\text{ for }w\in T,
\end{equation}
and such that $\tilde \psi\vert_{T_2}$ is a $\mathcal C^\infty$-approximation of $\psi\vert_{T_2}$. Define
\begin{equation*}
    j:W\rightarrow \R^n, \quad j(w):=i(w)+\tilde \psi(w)\cdot e_1, 
\end{equation*}
where $e_1$ is the first coordinate vector in $\R^n$. We claim $j$ satisfies the desired conditions.

It is clear that if $\tilde \psi\vert_W$ is sufficiently close to $\psi\vert_W\equiv 0$ in the $\mathcal C^\infty$-topology, then $j$ is a $\mathcal C^\infty$-approximation of $i$. We have $\pi_{1,2}\circ j_\C=\pi_{1,2}\circ i_\C$, so \eqref{lem2cond1'} is satisfied by the corresponding assumption about $i$. Let us then show that \eqref{lem2cond2'} is satisfied as well.

Let $w\in W_\C$ be a point such that $j_\C(w)\in \R^n$ and $\pi_2\circ j_\C(w)\in \pi_2\circ j(W)$. We need to show that $w\in W$.
 
The assumption that $\pi_2\circ j_\C(w)\in \pi_2\circ j(W)$ implies that
\begin{equation*}
    \pi_{1,2}\circ i_\C(w)=\pi_{1,2}\circ j_\C(w)\in \pi_{1,2}\circ j(W)=\pi_{1,2}\circ i(W),
\end{equation*}
so $w\in T$. Suppose that $w\not\in T_1$. By assumption there is some $w'\in W$ such that $\pi_2\circ j_\C(w)= \pi_2\circ j(w')$. Since $\psi(w)=4C$ and $\psi(w')=0$, projecting this equality onto the first coordinate we find that
\begin{gather*}
    (i_\C(w))_1+\tilde \psi(w)=(i_\C(w'))_1+\tilde \psi(w'), \\
    4C=\psi(w)=(i_\C(w'))_1-(i_\C(w))_1+\left(\tilde \psi(w')-\psi(w')\right)-\left(\tilde \psi(w)-\psi(w)\right)
\end{gather*}
Using \eqref{eqC} and \eqref{lem2eqApprox} we obtain that the absolute value of the right hand side is smaller than $4C$, which is a contradiction. Hence $w\in T_1$.

Suppose now that $w\not\in T_2$. Using \eqref{eq:imIneq}, \eqref{lem2eqApprox} and the fact that $\Im(\psi(w))=0$ we find that
\begin{equation*}
    \Vert \Im(j_\C(w))\Vert\geq \Vert \Im(i_\C(w))\Vert-| \Im(\tilde \psi(w))|\geq \delta -| \Im(\tilde \psi(w)-\psi(w))|>0.
\end{equation*}
This contradicts the assumption that $j_\C(w)\in \R^n$. Hence we must have $w\in T_2$.

Finally note that since $\tilde \psi\vert_{T_2}$ is close to $\psi\vert_{T_2}\equiv 0$ in the $\mathcal C^\infty$-topology, we have that $j_\C\vert_{T_2}$ is close to $i_\C\vert_{T_2}$ in the $\mathcal C^\infty$-topology. We can assume that they are so close that $j_\C\vert_{T_2}$ is an embedding, so in particular it is injective. Since $j_\C(w)\in \R^n$, we have that $j_\C(\bar w)=\overline{j_\C(w)}=j_\C(w)$. Hence we must have $w=\bar w$, i.e. $w\in W_\C\cap \R^m=W$. This finishes the proof.
\end{proof}

We will now strengthen the conclusion of Lemma \ref{lem2}, so that the projection of the entire singular locus of $X_\C$ is disjoint from the projection of $X^\ns$:
\begin{lem}\label{lem3}
Let $W\subset \R^m$ be a nonsingular compact algebraic set. Let $i:W\rightarrow \R^n$ be a smooth map, which embeds $W$ in $\R^n$ as a smooth manifold. Assume that $n\geq \dim W+2$. Then, there is a polynomial map $j:W\rightarrow \R^n$ arbitrarily close to $i$ in the $\mathcal C^\infty$-topology, such that if we define $X:=\overline{j(W)}^\Zar$ then the following conditions are satisfied:
\begin{eqenum}
    \item $X^\ns=j(W)$ and $j$ is a biregular isomorphism between $W$ and $X^\ns$, \label{lem3condIso}
    \item the map $\pi_{1,2}\vert_{X_\C}:X_\C\rightarrow \C^{n-2}$ is finite, \label{lem3condFinite}
    \item if we define $S\subset X_\C$ as the singular locus of $X_\C$ then $\pi_2(S)\cap \pi_2(X^\ns)=\emptyset$. \label{lem3condEmpty}    
\end{eqenum}
\end{lem}
\begin{proof}
Applying Lemma \ref{lem2} with the first two variables switched, we may assume that $i:W\rightarrow \R^n$ is a polynomial map such that if we define $Y:=\overline{i(W)}^\Zar$ then the following properties are satisfied:
\begin{eqenum}
    \item $Y^\ns=i(W)$ and $i$ is a biregular isomorphism between $W$ and $Y^\ns$, \label{lem3assumpt1}
    \item the map $\pi_{1,2}\vert_{Y_\C}:Y_\C\rightarrow \C^{n-2}$ is finite, \label{lem3assumpt2}
    \item $\pi_1(Y^\sing)\cap \pi_1(Y^\ns)=\emptyset$. \label{lem3assumpt3}
\end{eqenum}
Let $Z\subset Y_\C$ be the singular locus of $Y_\C$. By definition $Z\cap \R^n=Y^\sing$. Define the two following subsets of $Z$:
\begin{align*}
    R_1&:=\{z\in Z: \pi_{1,2}(z)\in \pi_{1,2}(Y^\ns)\},\\
    R_2&:=\{z\in R_1:\pi_2(z)\in \R^{n-1}\}.
\end{align*}
Due to \eqref{lem3assumpt1} the set $Y^\ns$ is compact, so thanks to \eqref{lem3assumpt2} the set $R_1$ is compact, and hence $R_2$, being closed in $R_1$, is compact as well. Both $R_1$ and $R_2$ are invariant under complex conjugation. Using compactness of the sets $R_1$ and $Y^\ns$ we find a constant $C>0$ such that
\begin{equation}\label{lem3eqC}
    |y_1|<C\text{ for }y\in R_1\cup Y^\ns,
\end{equation}
where $y_1\in \C$ denotes the first coordinate of a point $y\in \C^n$.

Define now the following subsets of $\C^{n-1}$:
\begin{equation*}
    T:=\pi_1(R_1)\cup \pi_1(Y^\ns),\quad T_1:=\pi_1(R_1),\quad T_2:=\pi_1(R_2).
\end{equation*}
We have that $T_2\subset T_1\subset T$ and all these three sets are compact and invariant under conjugation. 

We now make the important observation that $T_2$ is disjoint from $\pi_1(Y^\ns)$. Indeed, let $z\in R_2$ and assume that $\pi_1(z)\in \pi_1(Y^\ns)$. In particular we have $\pi_1(z)\in \R^{n-1}$. Moreover, from the definition of $R_2$ we have that $\pi_2(z)\in \R^{n-1}$, so in fact $z\in Z\cap \R^{n}=Y^\sing$. This contradicts \eqref{lem3assumpt3}. 

Let $T'$ be an open neighbourhood of $T_2$ in $T$ whose closure is disjoint from $\pi_1(Y^\ns)$, and which is invariant under complex conjugation. Define 
\begin{equation*}
    R':=(\pi_1\vert_{R_1})^{-1}(T').    
\end{equation*}
We have $T_2\subset T'\subset T_1$ and $R_2\subset R'\subset R_1$. The set $R_1\backslash R'$ is compact, and by the definition of $R_2$ its image $\pi_2(R_1\backslash R')\subset \C^{n-1}$ under $\pi_2$ is disjoint from $\R^{n-1}$. From compactness we deduce that there is a constant $\delta>0$ such that
\begin{equation}\label{lem3eqDelta}
    \Vert\Im(\pi_2(z))\Vert>\delta\text{ for }z\in R_1\backslash R'.
\end{equation}

Let $\theta:\C^{n-1}\rightarrow \C^{n-2}$ be the projection onto the last $n-2$ variables. We have that $\pi_{1,2}=\theta\circ\pi_1$. Since $\pi_{1,2}\vert_{Y_\C}$ is finite-to-one by \eqref{lem3assumpt2} and $T\subset\pi_1(Y_\C)$, we have that $\theta\vert_{T}$ is finite-to-one. Moreover
\begin{equation*}
    \theta(T)=\theta(\pi_1(R_1)\cup \pi_1(Y^\ns))\subset \pi_{1,2}(R_1) \cup \pi_{1,2}(Y^\ns)\subset \pi_{1,2}(Y^\ns)\subset \R^{n-2}.
\end{equation*}

Let $\psi:T\rightarrow \R$ be any continuous function, which is identically equal to zero on an open neighbourhood $P$ of $\pi_1(Y^\ns)$ in $T$ and which is identically equal to $4C$ on $T'$. Substituting $\frac{1}{2}(\psi(t)+\psi(\bar t))$ for $\psi(t)$ we may assume that it satisfies $\psi(\bar t)=\psi(t)=\overline{\psi(t)}$ for $t\in T$. Applying Lemma \ref{lem:functionApproxAK} (with $T,P,\theta$ and $\psi$ defined as in the last few paragraphs) we find a polynomial $\tilde \psi:\C^{n-1}\rightarrow \C$ defined over $\R$ such that 
\begin{equation}\label{lem3eqApprox}
    |\tilde \psi(t)-\psi(t)|<\min(\delta,C)\text{ for }t\in T
\end{equation}
and $\tilde\psi\vert_{\pi_1(Y^\ns)}$ is $\mathcal C^\infty$-close to $\psi\vert_{\pi_1(Y^\ns)}\equiv 0$. Define the polynomial map 
\begin{align*}
    f&:\C^n\rightarrow \C^n,\\
    f(x_1,\dots,x_n)&:=(x_1+\tilde \psi(x_2,\dots,x_n),x_2,\dots,x_n).
\end{align*}
We have that $f$ is defined over $\R$ and is a polynomial automorphism of $\C^n$ with the inverse given by
\begin{align*}
    f^{-1}&:\C^n\rightarrow \C^n,\\
    f^{-1}(x_1,\dots,x_n)&:=(x_1-\tilde \psi(x_2,\dots,x_n),x_2,\dots,x_n).
\end{align*}

Define $X:=f(Y)$ and $j:=f\circ i$. Since $\tilde \psi\vert_{\pi_1(Y^\ns)}$ is $\mathcal C^\infty$-close to zero, we have that $j$ is $\mathcal C^\infty$-close to $i$. Moreover, since $f$ is an automorphism of $\C^n$ it restricts to an isomorphism between $Y_\C$ and $X_\C$. It then further restricts to an isomorphism between $Y^\ns$ and $X^\ns$, so $j$ is a biregular isomorphism between $W$ and $X^\ns$ as desired. We also have that $\pi_{1,2}\vert_{X_\C}=\pi_{1,2}\vert_{Y_\C}\circ f^{-1}\vert_{X_\C}$ is finite, so it remains to show that \eqref{lem3condEmpty} is satisfied.

Let $s\in S$, where $S$ is the singular locus of $X_\C$. Suppose by contradiction that
\begin{equation}\label{lem3eqpiS}
    \pi_2(s)\in \pi_2(X^\ns).
\end{equation}
Define $z:=f^{-1}(s)$. Since $f^{-1}\vert_{X_\C}$ is an isomorphism onto $Y_\C$ we have $z\in Z$. The equation \eqref{lem3eqpiS} turns into 
\begin{equation}\label{lem3eqpiZ}
    \pi_2\circ f(z)\in \pi_2\circ f(Y^\ns).
\end{equation}
Projecting this onto the last $n-2$ coordinates we get
\begin{equation*}
    \pi_{1,2}(z)=\pi_{1,2}\circ f(z)\in \pi_{1,2}\circ f(Y^\ns)=\pi_{1,2}(Y^\ns).
\end{equation*}
Hence, by the definition of $R_1$, $z\in R_1$. We now consider two cases depending on whether $z\in R'$ or not. 

Assume that $z\in R'$. Let $y\in Y^\ns$ be such that 
\begin{equation}\label{lem3eqpiZy}
    \pi_2\circ f(z)= \pi_2\circ f(y).
\end{equation}
Since $\pi_1(z)\in T'$, we have that $\psi(\pi_1(z))=4C$. Similarly $y\in Y^\ns$, so $\psi(\pi_1(y))=0$. Projecting \eqref{lem3eqpiZy} onto the first coordinate we find that
\begin{gather*}
    z_1+\tilde \psi(\pi_1(z))=y_1+\tilde \psi(\pi_1(y))\\
    4C=\psi(\pi_1(z))= y_1-z_1+\left(\tilde \psi(\pi_1(y))-\psi(\pi_1(y))\right)-\left(\tilde \psi(\pi_1(z))-\psi(\pi_1(z))\right).
\end{gather*}
Using \eqref{lem3eqC} and \eqref{lem3eqApprox} we conclude that the absolute value of the right hand side is strictly smaller than $4C$. This is a contradiction.

Assume hence that $z\in R_1\backslash R'$. Note that $\pi_2\circ f(Y^\ns)=\pi_2(X^\ns)\subset \R^{n-1}$. From \eqref{lem3eqpiZ} we then deduce that
\begin{equation}\label{lem3eqpiZinRn}
    \pi_2\circ f(z)\in \R^{n-1}.
\end{equation}
On the other hand from \eqref{lem3eqDelta}, \eqref{lem3eqApprox} the definition of $f$ and the fact that $\psi(\pi_1(z))\in \R$ we get that
\begin{multline*}
    \Vert\Im(\pi_2\circ f(z))\Vert\geq \Vert \Im(\pi_2(z))\Vert - \left\Vert \Im\left(\pi_2(z)-\pi_2\circ f(z)\right)\right\Vert = \\
    = \Vert \Im(\pi_2(z))\Vert -\left|\Im\left(\tilde \psi(\pi_1(z))\right)\right|=\Vert \Im(\pi_2(z))\Vert -\left|\Im\left(\tilde \psi(\pi_1(z))-\psi(\pi_1(z))\right)\right|>0.
\end{multline*}
This contradicts \eqref{lem3eqpiZinRn}.

This shows that our assumption that $\pi_2(s)\in \pi_2(X^\ns)$ must be false and proves the desired condition \eqref{lem3condEmpty}.
\end{proof}

Theorem \ref{thm:main} now follows with little effort:
\begin{proof}[Proof of Theorem \ref{thm:main}]
After applying Lemma \ref{lem3} with the first two variables switched we may assume that the map $i:W\rightarrow \R^n$ is polynomial, and that if we define $Y:=\overline{i(W)}^\Zar$ then the following conditions are satisfied:
\begin{enumerate}
    \item $Y^\ns=i(W)$ and $i$ is a biregular isomorphism between $W$ and $Y^\ns$, \label{proofthmcondIso}
    \item the map $\pi_{1,2}\vert_{Y_\C}:Y_\C\rightarrow \C^{n-2}$ is finite, \label{proofthmcondFinite}
    \item if we define $S\subset Y_\C$ as the singular locus of $Y_\C$ then $\pi_1(S)\cap \pi_1(Y^\ns)=\emptyset$. \label{proofthmcondEmpty}    
\end{enumerate}
Since the projection $\pi_{1,2}$ factors through $\pi_1$, from \eqref{proofthmcondFinite} we deduce that $\pi_1\vert_{Y_\C}$ is finite as well. Hence $\pi_1(S)\subset \C^{n-1}$ is a complex algebraic set defined over $\R$. Let $p\in \R[x_2,\dots,x_n]$ be the sum of squares of generators of the ideal of real polynomials whose complexifications vanish on $\pi_1(S)$. As the set $\pi_1(Y^\ns)$ is disjoint from $\pi_1(S)$ and contained in $\R^{n-1}$, we have that $p$ is strictly greater than zero at all points of $\pi_1(Y^\ns)$. Let $U:=Y_\C \backslash (\C\times \Z(p))$, where $\Z(p)$ denotes the vanishing locus of $p$ in $\C^{n-1}$. By the definition of $p$ we have that $S\subset \C\times \Z(p)$, so $U$ is a nonsingular complex quasi-affine variety, which satisfies $U\cap \R^n=Y^\ns$.

Choose a small $\varepsilon>0$ and consider the following complex regular map defined over $\R$:
\begin{equation*}
    f:U\rightarrow \C^n,\quad f(y):=\left(y_1+\frac{\varepsilon}{p(y_2,\dots,y_n)},y_2,\dots,y_n\right),
\end{equation*}

We claim that this map is proper. To prove this, it suffices to show that if $(y^i)_{i=1}^\infty\in U$ is a sequence of points such that the sequence $f(y^i)$ of their images is bounded, then $(y^i)_{i=1}^\infty$ admits a subsequence convergent to a point $y\in U$. Let $(y^i)_{i=1}^\infty$ be such a sequence of points. Then the sequence $\pi_{1,2}(y^i)=\pi_{1,2}\circ f(y^i)$ is bounded. Since the map $\pi_{1,2}\vert_{Y_\C}$ is proper, after passing to a subsequence we may assume that $(y^i)_{i=1}^\infty$ converges to a point $y\in Y_\C$. We must have that $y\not \in \C\times \Z(p)$, for otherwise the first coordinate of $f(y^i)$ would escape to infinity as $i$ goes to infinity. 

Since proper morphisms are closed in the Zariski topology, it follows that $X_\C:=f(U)$ is an algebraic subset of $\C^n$. It is invariant under complex conjugation, hence defined over $\R$. We have that $f$ is an isomorphism from $U$ onto $X_\C$, with the inverse given by
\begin{equation*}
    f^{-1}:X_\C \rightarrow U,\quad f^{-1}(x):=\left(x_1-\frac{\varepsilon}{p(x_2,\dots,x_n)},x_2,\dots,x_n\right).
\end{equation*}
Since $U$ is nonsingular, it follows that $X_\C$ is nonsingular as well. As $f$ is defined over $\R$, it induces an isomorphism between the real loci of $U$ and $X_\C$. Therefore if we define $X:=f(U\cap \R^n)=f(Y^\ns)$ then $X$ is a nonsingular real algebraic subset of $\R^n$ and $f\vert_{Y^\ns}$ is a biregular isomorphism onto $X$. Moreover, since $Y^\ns=i(W)$ is Zariski dense in $Y_\C$ and hence also in $U$, we have that $X$ is Zariski dense in $X_\C$. It follows that $X_\C$ is the Zariski closure of $X$ in $\C^n$; in particular this closure is nonsingular.

Finally, if $\varepsilon$ is small then $f\vert_{Y^\ns}$ is arbitrarily close to the identity as a map into $\R^n$. It suffices to take $j:=f\vert_{Y^\ns}\circ i$.
\end{proof}
\section*{Acknowledgments}
The author is grateful to his mentor Wojciech Kucharz for guidance and support during his studies at the Jagiellonian University, and for help in preparing this and other manuscripts in real algebraic geometry.

The author was partially supported by the Polish National Science Center (NCN) under grant number 2024/53/N/ST1/02481.

\printbibliography

@article{akbulutApproximatingSubmanifoldsAlgebraic1992,
  title = {On approximating submanifolds by algebraic sets and a solution to the {{Nash}} conjecture},
  author = {Akbulut, S. and King, H.},
  year = 1992,
  month = dec,
  journal = {Invent Math},
  volume = {107},
  number = {1},
  pages = {87--98},
  issn = {1432-1297},
  doi = {10.1007/BF01231882}
}

@article{akbulutTranscendentalSubmanifoldsRn1993,
  title = {Transcendental submanifolds of {{Rn}}},
  author = {Akbulut, S. and King, H.},
  year = 1993,
  month = dec,
  journal = {Commentarii Mathematici Helvetici},
  volume = {68},
  number = {1},
  pages = {308--318},
  issn = {1420-8946},
  doi = {10.1007/BF02565821}
}

@article{bilskiApproximationPiecewiseregularMaps2020,
  title = {Approximation by piecewise-regular maps},
  author = {Bilski, Marcin and Kucharz, Wojciech},
  year = 2020,
  month = dec,
  journal = {Advances in Mathematics},
  volume = {375},
  pages = {107350},
  issn = {0001-8708},
  doi = {10.1016/j.aim.2020.107350}
}

@book{bochnakRealAlgebraicGeometry1998,
  title = {Real {{Algebraic Geometry}}},
  author = {Bochnak, Jacek and Coste, Michel and Roy, Marie-Fran{\c c}oise},
  year = 1998,
  publisher = {Springer},
  address = {Berlin, Heidelberg},
  doi = {10.1007/978-3-662-03718-8},
  copyright = {http://www.springer.com/tdm},
  isbn = {978-3-662-03718-8}
}

@article{bochnakVectorBundlesReal1989,
  title = {Vector bundles over real algebraic varieties},
  author = {Bochnak, J. and Buchner, M. and Kucharz, W.},
  year = 1989,
  month = may,
  journal = {K-Theory},
  volume = {3},
  number = {3},
  pages = {271--298},
  issn = {15730514, 09203036},
  doi = {10.1007/BF00533373}
}

@article{comessattiSullaConnessioneSuperficie1914,
  title = {{Sulla connessione delle superficie razionali reali}},
  author = {Comessatti, Annibale},
  year = 1914,
  month = dec,
  journal = {Annali di Matematica, Serie III},
  volume = {23},
  number = {1},
  pages = {215--284},
  issn = {0373-3114},
  doi = {10.1007/BF02419577}
}

@book{hartshorneAlgebraicGeometry1977,
  title = {Algebraic {{Geometry}}},
  author = {Hartshorne, Robin},
  year = 1977,
  series = {Graduate {{Texts}} in {{Mathematics}}},
  volume = {52},
  publisher = {Springer},
  address = {New York, NY},
  doi = {10.1007/978-1-4757-3849-0},
  copyright = {http://www.springer.com/tdm},
  isbn = {978-1-4757-3849-0}
}

@book{hirschDifferentialTopology1976,
  title = {Differential {{Topology}}},
  author = {Hirsch, Morris W.},
  year = 1976,
  series = {Graduate {{Texts}} in {{Mathematics}}},
  volume = {33},
  publisher = {Springer},
  address = {New York, NY},
  doi = {10.1007/978-1-4684-9449-5},
  copyright = {http://www.springer.com/tdm},
  isbn = {978-1-4684-9449-5}
}

@article{ivanovImprovementNashTognoliTheorem1984,
  title = {Improvement of the {{Nash-Tognoli}} theorem},
  author = {Ivanov, N. V.},
  year = 1984,
  month = jul,
  journal = {J Math Sci},
  volume = {26},
  number = {1},
  pages = {1642--1645},
  issn = {1573-8795},
  doi = {10.1007/BF01106439}
}

@book{mangolteRealAlgebraicVarieties2020,
  title = {Real {{Algebraic Varieties}}},
  author = {Mangolte, Fr{\'e}d{\'e}ric},
  year = 2020,
  series = {Springer {{Monographs}} in {{Mathematics}}},
  publisher = {Springer International Publishing},
  address = {Cham},
  doi = {10.1007/978-3-030-43104-4},
  copyright = {http://www.springer.com/tdm},
  isbn = {978-3-030-43104-4}
}

@article{nashRealAlgebraicManifolds1952,
  title = {Real {{Algebraic Manifolds}}},
  author = {Nash, John},
  year = 1952,
  journal = {Ann. Math.},
  volume = {56},
  number = {3},
  eprint = {1969649},
  eprinttype = {jstor},
  pages = {405--421},
  publisher = {[Annals of Mathematics, Trustees of Princeton University on Behalf of the Annals of Mathematics, Mathematics Department, Princeton University]},
  issn = {0003-486X},
  doi = {10.2307/1969649}
}

@book{rudinRealComplexAnalysis1974,
  title = {Real and {{Complex Analysis}}},
  author = {Rudin, Walter},
  year = 1974,
  publisher = {McGraw-Hill Science, Engineering \& Mathematics}
}

@article{seifertAlgebraischeApproximationMannigfaltigkeiten1936,
  title = {{Algebraische Approximation von Mannigfaltigkeiten.}},
  author = {Seifert, H.},
  year = 1936,
  journal = {Math. Z.},
  volume = {41},
  pages = {1--17},
  issn = {0025-5874; 1432-1823}
}

@article{tognoliSuCongetturaDi1973,
  title = {{Su una congettura di Nash}},
  author = {Tognoli, A.},
  year = 1973,
  journal = {Ann. Della Scuola Norm. Super. Pisa - Sci. Fis. E Mat.},
  volume = {27},
  number = {1},
  pages = {167--185},
  issn = {0036-9918}
}
\end{document}